\documentclass[11pt]{article}
\usepackage{amsfonts,amsmath,tikz}
\usepackage{epsfig}
\usepackage{latexsym}
\usepackage{color}
\usepackage{soul}  
\usepackage{epsf,amssymb}
\usepackage{amsthm}
\usepackage{yfonts}
\usepackage{hyperref}
\usepackage{ifthen}
\usepackage{enumitem}

\usepackage{float} 

\newtheorem{theorem}{Theorem}
\newtheorem{lemma}[theorem]{Lemma}

\newtheorem{corollary}[theorem]{Corollary}
\newtheorem{proposition}[theorem]{Proposition}

\newtheorem{definition}[theorem]{Definition}

\def\vertex(#1){\put(#1){\circle*{2}}}
\def\vertexo(#1){\put(#1){\circle{2}}}
\def\vert(#1){\put(#1){\circle*{1.5}}}
\def\verto(#1){\put(#1){\circle{1.5}}}
\def\lab(#1)#2{\put(#1){\makebox(0,0)[c]{#2}}}
\definecolor{darkgreen}{RGB}{0,100,0}

\newcommand{\epn}{\mbox{epn}}

\newcommand{\gt}{\gamma_t}

\newcommand{\gammakL}{\gamma_k^{\mathrm L}}
\newcommand{\gammaenaL}{\gamma_1^{\mathrm L}}
\newcommand{\diam}{\mathrm{diam}}

\usepackage{todonotes}

\tikzset{My Style/.style={draw, circle, fill=black,scale=0.5}} 

\tikzset{My Style2/.style={draw, circle, fill=white,scale=0.5}} 

\title{On $k$-limited domination in graphs}

\author{
Dragana Bo\v zovi\'{c} \\
\small \it University of Maribor, FEECS, Koro\v ska cesta 46, 2000 Maribor, Slovenia \\
\small \it University of Maribor, FNM, Koro\v ska cesta 160, 2000, Maribor, Slovenia
\small \tt dragana.bozovic@um.si\\
\and
Gordana Radi\'{c} \\
\small \it University of Maribor, FEECS, Koro\v ska cesta 46, 2000 Maribor, Slovenia \\
\small \it IMFM, Jadranska ulica 19, 1000 Ljubljana, Slovenia \\
\small \tt gordana.radic@um.si
\and 
\v{Z}ana Kovijani\'{c}-Vuki\'{c}evi\'{c} \\
\small \it University of Montenegro, PMF, 81000 Podgorica, Montenegro \\
\small \tt zanak@ucg.ac.me
\and
Aleksandra Tepeh \\
\small \it University of Maribor, FEECS, Koro\v ska cesta 46, 2000 Maribor, Slovenia \\
\small \it IMFM, Jadranska ulica 19, 1000 Ljubljana, Slovenia \\
\small \it University of Maribor, FNM, Koro\v ska cesta 160, 2000, Maribor, Slovenia\\
\small \tt aleksandra.tepeh@um.si \\
}
\date{}

\begin{document}
\maketitle

\begin{abstract}
In this work, we introduce and study the notion of \emph{$k$-limited domination} in graphs, 
motivated by applications where dominating vertices have bounded capacity and cannot be 
overloaded by too many external neighbors. 
Formally, given an integer $k \le \Delta(G)$, a set of vertices $D \subseteq V(G)$ is called 
a $k$-limited dominating set if it is a dominating set and, in addition, each vertex of $D$ 
has at most $k$ neighbors outside $D$. The minimum cardinality of such a set is the 
\emph{$k$-limited domination number}, denoted by $\gammakL(G)$. 
Since $\Delta(G)$-limited domination coincides with the classical domination, we restrict our attention to the nontrivial range $1 \le k < \Delta(G)$, where the degree limitation becomes meaningful and leads to new combinatorial phenomena.

In this paper, we initiate the study of this concept by deriving sharp general bounds for $\gammakL(G)$ and identifying conditions under which these bounds can be further improved. We establish a connection between $k$-limited domination and $(1,t)$-domination. In particular, for $d$-regular graphs we prove that $\gammakL(G)=\gamma_{1,d-k}(G)$. In the special case $k=1$, we show that $1$-limited domination is tightly linked to graph packings, yielding the bound $\gamma^{\mathrm L}_1(G) \le n - \rho(G)$ and its characterization.

The study reveals several natural open questions and indicates that limited domination provides a rich ground for further research.
\end{abstract}


\section{Introduction}

Domination is a central concept in graph theory, capturing the idea of placing the minimal number of resources (e.g., guards, transmitters, or facilities) so that every point in the network is either a resource itself or adjacent to one. Formally, a \emph{dominating set} in a graph \(G=(V(G),E(G))\) is a subset \(D\subseteq V(G)\) such that every vertex in \(V(G)-D\) has at least one neighbor in \(D\). The \emph{domination number} \(\gamma(G)\) is the minimum size of such a set.
For a comprehensive treatment of domination concepts, including historical insights into their introduction, examples of applications, and an in-depth discussion of their properties and interrelations, we refer the reader to three excellent monographs \cite{HHH1,HHH3,HHH2}.

For modeling real-world scenarios, several variations of domination have been introduced, like the \textit{total domination} where every vertex including those in the dominating set must be adjacent to a dominator \cite{MA_Henning}. In the \textit{efficient domination}, it is required that each vertex is dominated by exactly one dominator \cite{HHH4, kpy}. In $(1,t)$-domination each non-dominator must be adjacent to at least one dominator, while each dominator must in turn dominate at least $t$
other dominators \cite{fakhran, favaron}.
We postpone the precise definitions of these variants until they are needed later. At this point, we only recall the formal definition of $k$-domination, since our new concept of $k$-limited domination is, in some sense, a dual to $k$-domination.

Let $k$ be a positive integer. A subset $D\subseteq V(G)$ is a \textit{$k$-dominating set} if every vertex $v\in V(G)- D$ has at least $k$ neighbors in $D$, \cite{Fink}. The \textit{$k$-domination number}, $\gamma_k(G)$, is the minimum size of such a set. Note that $\gamma_1(G)=\gamma(G)$.
In this work we introduce and study the notion of $k$-limited domination in graphs, which focuses on the constraint on how many non-dominators a dominator can reach, rather than the reverse.

Given a nonnegative integer $k \le \Delta(G)$, a set of vertices $D$ in a graph $G$ is called a \emph{$k$-limited dominating set} if it is a dominating set and, in addition, each vertex in $D$ has at most $k$ neighbors outside $D$. We denote by $\gammakL(G)$ the minimum cardinality of such a set.

The additional degree constraint in the definition of $k$-limited domination seems natural already from a theoretical perspective, as it incorporates the classical domination as a special case: if $k=\Delta(G)$, the constraint is trivially satisfied for every dominating set, so $\gamma^{\mathrm L}_{\Delta(G)}(G)=\gamma(G)$.
Moreover, its is also natural from a practical point of view: while classical domination aims to cover all vertices with the smallest possible number of vertices (resources), the $k$-limited domination model reflects a limitation that frequently appears in real-world systems, namely, avoiding the overloading of certain resources or vertices.
For instance, one broad class of applications arises in service allocation problems, where certain facilities or providers (e.g., warehouses, distribution centers, hospitals, clinics, veterinary services) serve surrounding demand points.
In classical domination, a single facility could, in principle, cover an unlimited number of demand points.
In practice, however, each facility or provider has capacity limits, such as storage space, available personnel, consultation time, travel distance, or the quantity of goods and services that can be delivered in a given period.
The $k$-limited domination model captures these constraints by ensuring that no facility or provider is responsible for more than $k$ distinct demand points, leading to allocations that are both realistic and implementable.

The aim of our paper is not only to formally introduce this practically relevant concept, but also to initiate a systematic investigation and provide a foundation for its thorough study in the future.
 
\medskip 

The rest of the paper is organized as follows. In the next section, we provide formal definitions of the
concepts under consideration, together with exact formulae for $\gammakL(G)$ of certain graph classes, 
which will later be used also for establishing the tightness of our bounds. These are considered in Section~3, where we first establish general bounds for arbitrary graphs, and demonstrate their tightness by  
focusing on efficient graphs. For such graphs we provide a sufficient condition under which the $k$-limited domination number coincides with the classical domination number. Furthermore, under an additional structural assumption, we derive an alternative upper bound expressed in terms of $\gamma(G)$, which in certain cases improves upon the general upper bound. In Section~4, we bound $k$-limited domination in terms of $(1,t)$-domination. As an interesting consequence, we obtain that if $G$ is a $d$-regular graph and $k<d$, then $\gammakL(G)=\gamma_{1,d-k}(G)$.
This means, for instance, that in cycles the $1$-limited domination number coincides with the total domination number, while in cubic graphs $\gamma_{2}^{\mathrm{L}}(G)$ equals the total domination number, and $\gamma_{1}^{\mathrm{L}}(G)$ coincides with $(1,2)$-domination number.  
In Section~\ref{sec:1L}, we take a closer look at $1$-limited domination. We characterize graphs that attain the general lower bound, as well as the improved upper bound expressed in terms of the packing number of a graph. In particular, we establish a Gallai-type inequality, $\gamma^{\mathrm L}_1(G) + \rho(G) \le n$, and examine the cases when the packing number $\rho(G)$ attains its minimum and maximum possible value.
Finally, Section~6 provides concluding remarks and outlines several directions for future research.

\section{Preliminaries and basic observations}

We begin by recalling standard notation and terminology. Let $G = (V(G),E(G))$ be a finite simple undirected graph with $|V(G)| = n$. For a vertex $v \in V(G)$, we denote by $N(v)$ its open neighborhood and by $N[v] = N(v) \cup \{v\}$ its closed neighborhood. Similarly, $N(S)$ denotes the set of neighbors of vertices of $S$ in $V(G)- S$. The degree of $v\in V(G)$ is denoted by $\deg(v)$, and the minimum and maximum degrees of $G$ by $\delta(G)$ and $\Delta(G)$, respectively.  A vertex of degree $1$ is called a \emph{pendant vertex}, while any vertex that is adjacent to at least one pendant vertex is called a \emph{support vertex}.
The distance between two vertices $u,v\in V(G)$, denoted by $d(u,v)$, is the length of a shortest $u-v$ path in $G$, and the diameter of a connected graph $G$, $\diam(G)$, is the maximum distance between any pair of vertices of $G$. 

Recall that a \textit{dominating set} of a graph $G$ is a subset $D$ of $V(G)$ such that every vertex not in $D$ is adjacent to some vertex in $D$. The \textit{domination number}, $\gamma(G)$, is the minimum cardinality of a dominating set of $G$. A dominating set of $G$ with cardinality equal to $\gamma(G)$ is called a $\gamma(G)$-set.


\begin{definition}
\label{gldef}
Let $G$ be a graph and let $k$ be an integer such that $0\leq k \le \Delta(G)$.  
A set $D \subseteq V(G)$ is called a \emph{$k$-limited dominating set} if $D$ is a dominating set, and
for every $u \in D$ it holds that $|N(u)-D| \le k$.
\end{definition}

The \emph{$k$-limited domination number} of $G$, denoted by $\gammakL(G)$, is the minimum cardinality of a $k$-limited dominating set in $G$. A $k$-limited dominating set of order $\gammakL(G)$ will be referred to as $\gammakL(G)$-set. For convenience, we shall refer to the second condition in  Definition \ref{gldef} as the \emph{$k$-limited constraint}.

Let $D$ be a $k$-limited dominating set of a graph $G$.
Since $|N(u)-D| \le \Delta(G)$ for all $u \in D$, every dominating set is trivially a $\Delta(G)$-limited dominating set, hence $\gamma_{\Delta(G)}^L(G) = \gamma(G)$.
If $k=0$, then every vertex in $D$ has no neighbors outside $D$, which forces $D = V(G)$, and therefore $\gamma_0^{\mathrm L}(G) = |V(G)|$.
Thus, the interesting range for $k$ is $1 \le k \le \Delta(G)-1$ for graphs with $\Delta(G) \ge 2$. 
In addition, if $G$ has connected components $G_1,\dots,G_r$, then
$\gammakL(G) = \sum_{i=1}^r \gammakL(G_i)$. Therefore, in what follows we restrict our attention to connected graphs on at least $3$ vertices.

Note that when $k<\Delta(G)$, the restriction that each dominating vertex has at most $k$ external neighbors may increase the size of a minimum dominating set, so $\gamma(G)\le \gamma^{\mathrm L}_k(G)$, with equality still possible in some graphs.

\medskip

To better understand the newly introduced concept, we now
consider some basic families of graphs and their 
$k$-limited domination numbers (for the relevant values of $k$), where we use the following standard notation: $P_n$, $S_n$ and $K_n$ are the path, the star and the complete graph, respectively, on $n$ vertices, and $K_{m,n}$ denotes the complete bipartite graph with 
partite sets of sizes $m$ and $n$.

\begin{proposition}
\label{families}
Let $k$ be a positive integer. Then the following holds.
\begin{enumerate}[label=(\roman*)]
\item For every $n\geq 3$ we have $\gammaenaL(P_n) = \lceil\frac n2\rceil$.
\item For every $n\geq 4$ and $k$ such that $1\leq k<n$ we have $\gammakL(S_n) = n - k$. 
\item\label{polni} For every $n\geq 3$ and $k$ such that $1\leq k<n$ we have $\gammakL(K_n) = n - k$.

\item For every $m$ and $n$ such that $2\leq m\leq n$ and $1\leq k< n$, we have
$$\gammakL(K_{m,n})=\left\{\begin{array}{ccl}
1+n-k;&m\leq k\\
m+n-2k;&k<m.
\end{array}\right.$$
\end{enumerate}
\end{proposition}

\begin{proof}
\textit{(i)} Let $D\subseteq V(P_n)$ be a $\gammaenaL(P_n)$-set where $P_n=x_1x_2\ldots x_n$. Since every vertex in $V(P_n)-D$ must be dominated, and each vertex in $D$ can dominate at most one of its neighbors, it follows that $|V(P_n)-D|\leq|D|$. Thus, $n=|D|+|V(P_n)-D|\leq2|D|$, and so, $\gammaenaL(P_n)=|D|\geq\frac{n}{2}$. To show the opposite inequality, 
consider the set $S=\{x_i\in V(P_n); i\equiv 1\,(\text{mod }4)\text{ or }i\equiv 0\,(\text{mod }4)\}$. If $n$ is even, let $D'=S$, otherwise $D'=S\cup\{x_n\}$.  Clearly, $D'$ is a $1$-limited dominating set of $P_n$ of cardinality $|D'|=\lceil\frac n2\rceil$. Hence, $\gammaenaL(P_n)\leq|D'|= \lceil\frac n2\rceil$.

\medskip
\textit{(ii)} Let $c\in V(S_n)$ be the central vertex of the star $S_n$. Consider the set $D\subseteq V(S_n)$, consisting of $c$ and $n-k-1$ leaves. Then, $D$ is a $k$-limited dominating set, which implies that $\gammakL (S_n)\leq n-k$.

To derive the opposite inequality let $D'$ be a $\gammakL(S_n)$-set. If $c$ is not in $D'$, then every other vertex is in $D'$, thus $|D'|=n-1$. However, this is clearly a contradiction unless $k=1$. So, if $k\geq 2$, then $c\in D'$, and at least $n-k-1$ leaves must be contained in $D'$ as well, otherwise a contradiction with $k$-limited constraint is obtained. In the last case when $k=1$ and $c\in D'$, we immediately derive $|D'|\geq n-1$. Thus $|D'|\geq n-k$ for every $k\geq 1$.

\medskip
\textit{(iii)} Choose an arbitrary subset $D \subseteq V(K_n)$ of size $n - k$. Since $D$ is clearly a $k$-limited dominating set, we have
$\gammakL(K_n)\leq  n-k$. To show that equality holds, suppose to the contrary that 
there exists a $k$-limited dominating set $D'$ with $|D'| \leq n - k - 1$. But then $|V(K_n) - D'| \geq k + 1$, yielding a contradiction  with the $k$-limited constraint. Hence, $\gammakL(K_n) = n - k$. 

\medskip

\textit{(iv)} Let $G = K_{m,n}$ and let $(V_1,V_2)$ be a bipartition of $V(G)$ such that $|V_1| = m$ and $|V_2| = n$. 
Assume first that $m \leq k$. Let $D \subseteq V(G)$ be a set consisting of $n-k$ vertices from $V_2$ together with exactly one vertex from $V_1$. Clearly, $D$ is a $k$-limited dominating set. Hence, $\gammakL(G) \leq 1+n-k$. Assuming that $\gammakL(G) \leq n-k$, there exists a $k$-limited dominating set $D'$ with $|D'| \leq n-k$. If $ |D'\cap V_1|\geq1$, there exist at least $k+1$ vertices in $V_2$ that are not in $D'$. This means that any vertex in $ D'\cap V_1$ has at least $k+1$ neighbors outside $D'$, a contradiction. On the other hand, if $|D' \cap V_1| = 0 $, then every vertex of $V_2$ must belong to $D'$, otherwise some vertex of $V_2$ would not be dominated. In this case, $|D'| = |V_2| = n$, a contradiction again. We therefore obtain that $\gammakL(G) = n-k+1$.

Now suppose instead that $k < m$. Let $D \subseteq V(G)$ consist of $n-k$ vertices from $V_2$ and $m-k$ vertices from $V_1$. Then $D$ is clearly a $k$-limited dominating set, which gives $\gammakL(G) \leq  m+n-2k$.
To prove the opposite inequality, let $D'$ be any $k$-limited dominating set. Note that $D'$ must contain vertices from both, $V_1$ and $V_2$.
Take $u\in V_1\cap D'$. Since $u$ is adjacent to all $n$ vertices of $V_2$, we have
$|N(u)-D'|=n-|V_2\cap D'|\leq k$, yielding $|V_2\cap D'|\geq n-k$.
Similarly, for any $v\in V_2\cap D$,
$|N(v)-D'|=m-|V_1\cap D'|\leq k$ and thus $|V_1\cap D'| \geq m-k$. 
Adding the two inequalities gives
$$|D'|=|V_1\cap D'|+|V_2\cap D'| \geq (m-k)+(n-k)=m+n-2k.$$
Therefore we conclude that $\gamma_k^{\mathrm L}(K_{m,n})=m+n-2k$ for $k<m$.
\end{proof}

\section{Bounds}

We start this section by introducing the general bounds and discussing both their tightness and the conditions that allow for their improvement.

\begin{theorem} 
\label{bounds}
Let $G$ be a connected graph of order $n\geq 3$, and $k < \Delta(G)$. Then
$$
\max\left\{\gamma(G),\ \left\lceil \frac{n}{k+1} \right\rceil\right\} 
\leq \gammakL(G)  \leq 
n-k.
$$
\end{theorem}

\begin{proof}
Let $D \subseteq V(G)$ be a $\gammakL(G)$-set, i.e.~$|D| = \gammakL(G)$. Since every vertex in $V(G) - D$ must be adjacent to some vertex in $D$, and each $u \in D$ can dominate at most $k$ neighbors in $V(G) - D$, we obtain:
$$
n = |V(G)| = |D| + |V(G) - D| \leq |D| + k|D| = |D|\;(k+1),
$$
which implies
$$
\gammakL(G) \geq \left\lceil \frac{n}{k+1} \right\rceil.
$$
Furthermore, since every $k$-limited dominating set is in particular a dominating set, we have $\gammakL(G) \geq \gamma(G)$, and the lower bound follows.

To prove the upper bound, let $v \in V(G)$ be a vertex of maximum degree. Since $k < \deg(v)$, we may choose a subset $S \subseteq N(v)$ of exactly $k$ neighbors of $v$. Then the set $D= V(G)-S$ is a dominating set, and each vertex in $D$ has at most $k$ neighbors in $S=V(G)-D$.
Thus, $D$ is a $k$-limited dominating set of size $n - k$, which shows $\gammakL(G) \leq n - k$.
\end{proof}

The lower bounds in the above theorem, $\gamma(G)$ and $\lceil \frac{n}{k+1} \rceil$, are incomparable. Clearly, for a star on $n$ vertices we have $\lceil \frac{n}{k+1}\rceil \ge 2$, while its domination number is $1$. 
On the other hand, if $G$ is a graph on $n=2t$, $t\geq 3$, vertices, obtained from an arbitrary graph on $t$ vertices by attaching one pendant vertex to each of its vertices, we have $\gamma(G)=t>\lceil \tfrac{2t}{k+1} \rceil$ if $k\geq2$.

\medskip

A natural question concerning the obtained bounds is also whether they can be achieved, and if so, by which graphs.
For special graphs, known as efficient graphs, we now present a sufficient condition under which the equality $\gammakL(G) = \gamma(G)$ holds.

A set $S \subseteq V(G)$ in a graph $G$ is called an \emph{efficient dominating set} (also called a \textit{perfect code} \cite{Biggs}) if every vertex $v \in V(G)$ is dominated by exactly one vertex in $S$, that is, $|N[v] \cap S| = 1$. 
A graph that admits at least one efficient dominating set is called an \emph{efficient graph}, in literature known also as an \emph{efficient closed domination graph} \cite{kpy}. Not every graph is efficient; for example, it is easy to see that the cycles $C_4$ and $C_5$ are not efficient graphs. In fact, for $n\geq 3$ a cycle $C_n$ is efficient if and only if $n\equiv 0 \; (\text{mod } 3)$, while all paths are efficient (for more examples of efficient graphs see \cite{HHH3}).
Bange et al.~\cite{Bange} showed that if $S$ is an efficient dominating set of $G$, then it is a $\gamma(G)$-set.

For the purposes of studying $k$-limited domination in efficient graphs, we introduce the following notation: let $S$ be an efficient dominating set of $G$, $M(S)=\max\{\deg(v);\ v\in S\}$, and $m(G)=\min\{M(S);\, S\text{ is an efficient dominating set of }G\}$.

We will show that for efficient graphs, once $k$ reaches or exceeds $m(G)$, i.e.~the smallest possible maximum degree of the vertices among all efficient dominating sets, the $k$-limited domination number attains its minimum possible value $\gamma(G)$ and remains constant for all larger values of $k$.

\begin{proposition} 
\label{efficient}
Let $G$ be a connected efficient graph of order $n\geq 3$. If $m(G)\leq k<\Delta(G)$, we have
$\gammakL(G)=\gamma(G)$.
\end{proposition}

\begin{proof}
Let $S$ be an efficient dominating set of a connected graph $G$ with $M(S)=m(G)$. 
By the assumption, the closed neighborhoods $\{N[u];\ u\in S\}$ form a partition of $V(G)$, hence for each $u\in S$ we have $N(u)\cap S=\emptyset$ and thus $|N(u)-S|=\deg(u)\le M(S)=m(G)\le k$. 
So $S$ is a $k$-limited dominating set, which gives $\gammakL(G)\le |S|=\gamma(G)$. The opposite inequality holds
by Theorem \ref{bounds}, hence $\gammakL(G)=\gamma(G)$. 
\end{proof}

Proposition~\ref{efficient} also shows that the lower bound $\gamma(G)$ in Theorem~\ref{bounds} is tight. Among efficient graphs, one can also easily find graphs for which the lower bound 
$\lceil \frac{n}{k+1} \rceil$ 
in Theorem \ref{bounds}
is attained. A simple example of such a graph is $K_{k+2}$, where $\gamma(K_{k+2})=1 < \gammakL(K_{k+2})= \lceil \frac{k+2}{k+1}\rceil=2$.



\medskip
The tightness of the upper bound $n-k$ in Theorem \ref{bounds} is attained for stars and complete graphs, see Proposition \ref{families}. We summarize the above observations in the following corollary.

\begin{corollary}
The bounds in Theorem~\ref{bounds} are tight.
\end{corollary}

The general upper bound $n - k$ from Proposition \ref{bounds} can, in some cases, be significantly improved.  
In particular, if in a graph $G$ there exists a $\gamma(G)$-set that satisfies an additional structural property, we can bound $\gammakL(G)$ in terms of $\gamma(G)$ itself.  
This is formalized in the following proposition.



\begin{proposition}
\label{Delta-v-zg}
Let $G$ be a connected graph of order $n\geq 3$ and let $k$ be an integer such that $1 \le k < \Delta(G)$. 
Suppose there exists a $\gamma(G)$-set $D$ of $G$ such that, for every $u\in D$ with $\deg(u)>k$, 
there exist at least $\deg(u)-k$ pairwise distinct neighbors $w$ of $u$ with $\deg(w)\le k+1$. Then we have the following tight bound 
$$
\gammakL(G) \leq (\Delta(G)-k+1)\,\gamma (G).
$$
\end{proposition}

\begin{proof} 
Let $D$ be a $\gamma(G)$-set satisfying the assumption of the proposition.
For each vertex $u\in D$ we proceed as follows. If $\deg(u)\le k$, we set $L_u=\emptyset$,
otherwise, we choose a subset $L_u\subseteq N(u)$ consisting of exactly $\deg(u)-k$ distinct neighbors of $u$ whose degrees are at most $k+1$. 
This selection is possible by the assumption on $D$. Note that we do not require the sets $L_u$ to be mutually  disjoint, nor disjoint from $D$.
Let $$D' = D \cup \bigcup_{u\in D} L_u.$$
Since $D\subseteq D'$ and $D$ dominates $G$, $D'$ also dominates $G$. For vertices of $D'$ we now verify the $k$-limited constraint. 
For $u\in D$ with $\deg(u)\le k$ we have $|N(u)-D'|\leq \deg(u)\le k$.
For $u\in D$ with $\deg(u)>k$, it holds $L_u\subseteq N(u)\cap D'$ and $|L_u|=\deg(u)-k$, hence
$$|N(u)-D'|=\deg(u)-|N(u)\cap D'|\leq \deg(u)-|L_u|=k.$$
Finally, for any $v\in L_u$ we have $\deg(v)\le k+1$ and $u\in N(v)\cap D'$, thus
$|N(v)-D'|\leq \deg(v)-1\leq k$.
Thus $D'$ is a $k$-limited dominating set and for its size we derive
\begin{align*}
|D'| & =|D \cup \bigcup_{u\in D} L_u|
\leq  |D| + \sum_{u\in D}|L_u|\leq |D| + \sum_{u\in D} (\deg(u)-k) \\
& \le |D| + \sum_{u\in D} (\Delta(G)-k)
= (\Delta(G)-k+1)\,|D| \\
& = (\Delta(G)-k+1)\,\gamma(G).
\end{align*}

\noindent Therefore $\gammakL(G)\le |D'|\le (\Delta(G)-k+1)\,\gamma(G)$.
The bound is tight, as by Proposition \ref{families} it is attained for complete graphs and stars. 
\end{proof}

If $G$ is a connected graph satisfying the assumption in Proposition \ref{Delta-v-zg}, then the natural question is, which of the upper bounds $n-k$ and $(\Delta(G)-k+1)\,\gamma(G)$ provides the stronger estimate. It turns out that the two bounds are in fact incomparable, as neither dominates the other in general. To illustrate this phenomenon, we consider the following family of graphs.

A \textit{broom} is a graph obtained from a path by attaching a star at one of its endvertices. Formally, a broom $B(p,m)$ of order $n=p+m$ is constructed by taking a path $P_p$ on $p$ vertices and attaching $m$ pendant vertices to one end of this path.

Specifically, examine the broom $G=B(5,m)$ with $m\geq 3$, $n = 5+m$ and $\Delta(G) = m+1$.  
Note that $\gamma(G) = 2$.  For such graphs, both expressions $n-k$ and $(\Delta(G)-k+1)\gamma(G)$ can be evaluated explicitly, and the outcome is summarized in Table~\ref{table:broom}. In the first case, when $k=m-2$, we observe that the bound $n-k$ yields the smaller value, in the second case, for $k=m-1$, both bounds coincide, while in the third case, when $k=m$, the bound $(\Delta(G)-k+1)\gamma(G)$ is sharper. Thus, even within this restricted graph family, the bounds are incomparable, as shown in Table \ref{table:broom}.

\begin{table}[H]
\centering
\begin{tabular}{|c|c|c|}
\hline
$k$ & $n-k$ & $(\Delta(G)-k+1)\gamma(G)$ \\ \hline
$m-2$ & \textbf{7} & 8 \\ \hline
$m-1$ & \textbf{6}  & \textbf{6}\\ \hline
$m$ & 5 & \textbf{4} \\ \hline
\end{tabular}
\caption{Comparison of the two bounds for the broom $G=B(5,m)$ with $m \geq 3$.}
\label{table:broom}
\end{table}

We can therefore summarize our observations as follows.

\begin{corollary}
Let $G$ be a connected graph of order $n\geq 3$ and let $k$ be an integer such that $1 \le k < \Delta(G)$. 
Suppose there exists a $\gamma(G)$-set $D$ of $G$ such that for every $u\in D$ with $\deg(u)>k$
there exist at least $\deg(u)-k$ pairwise distinct neighbors $w$ of $u$ with $\deg(w)\le k+1$. Then we have the tight bound $$\gammakL(G) \leq \min\{n-k,\;(\Delta(G)-k+1)\gamma (G)\}.$$
\end{corollary}

\medskip

\section{$(1,t)$-domination and regular graphs}

In this section, we shall bound the $k$-limited domination number in terms of $(1,t)$-domination, which yields interesting consequences for regular graphs.

Let $t \ge 1$ be an integer and let $G$ be a graph with $\delta(G) \ge t$.  
A set $S \subseteq V(G)$ is a \emph{$(1,t)$-dominating set} if every vertex in $V(G)- S$ has at least one neighbor in $S$, and every vertex in $S$ has at least $t$ neighbors in $S$. The minimum cardinality of such a set is the \emph{$(1,t)$-domination number}, denoted by $\gamma_{1,t}(G)$, and any $(1,t)$-dominating set of size $\gamma_{1,t}(G)$ is called a \emph{$\gamma_{1,t}(G)$-set}. For the special case $t=1$, the $(1,1)$-domination number coincides with the \emph{total domination number} $\gamma_t(G)$, which is defined as the minimum cardinality of a set $S \subseteq V(G)$ such that every vertex of $G$ has a neighbor in $S$.

\medskip
Since $k<\Delta(G)$, we have $\Delta(G)-k>0$, thus we can consider $(1,\Delta(G)-k)$-domination for $k\geq \Delta(G)-\delta(G)$.

\begin{proposition}
\label{zgornja-1t}
Let $G$ be a connected graph of order $n\geq 3$ and let $k$ be an integer such that $\Delta(G)-\delta(G)\leq k < \Delta(G)$.
Then $\gammakL(G)\leq \gamma_{1,\Delta(G)-k}(G)$.
\end{proposition}

\begin{proof}
Let $t=\Delta(G)-k$. Then $0<t\leq \delta(G)$, which ensures that $(1,t)$-domination 
is well-defined.  
Let $D \subseteq V(G)$ be a $\gamma_{1,t}(G)$-set. This means that every vertex in $V(G)-D$ has at least one neighbor in $D$, and every vertex $v\in D$ has at least $t$ neighbors in $D$. The latter implies that $v$ has at most $\Delta(G)-t=k$ neighbors in $V(G)-D$. Therefore $D$ is also a $k$-limited dominating set of $G$, which implies $\gammakL(G)\leq|D|=\gamma_{1,t}(G)$.
\end{proof}

Similarly, under a suitable degree condition we can derive a lower bound for $\gammakL(G)$ via $(1,t)$-domination.

\begin{proposition}
\label{spodnja-1t}
Let $G$ be a connected graph of order $n\geq 3$ and let $k$ be an integer with $1 \le k < \Delta(G)$. 
If $\delta(G) \ge k+1$, then
$\gammakL(G) \geq 
\gamma_{1,\,\delta(G)-k}(G)$.
\end{proposition}

\begin{proof}
Set $t=\delta(G)-k$. The assumption $\delta(G)\ge k+1$ yields $\delta(G) \geq t\ge 1$, so $(1,t)$-domination is well-defined. 
Let $D$ be a minimum $k$-limited dominating set.
Then vertices in $V(G)-D$ have a neighbor in $D$. Furthermore,
for each $u\in D$, $|N(u)- D|\le k$, hence $|N(u)\cap D|\ge \text{deg}(u)-k \ge \delta(G)-k=t$. 
Therefore $D$ is a $(1,t)$-dominating set and $\gamma_{1,\,\delta(G)-k}(G)\le |D|=\gammakL(G)$.
\end{proof}

In the case of regular graphs the degree conditions in Propositions 
\ref{zgornja-1t} and \ref{spodnja-1t}
are automatically satisfied, so the two bounds coincide, yielding the exact formula for $\gammakL(G)$ in terms of $(1,t)$-domination.  
This simultaneously establishes the tightness of the bounds given in Propositions~\ref{zgornja-1t} and ~\ref{spodnja-1t}.

\begin{corollary}
\label{d-regular}
Let $G$ be a connected $d$-regular graph of order $n\geq 3$ and let $k$ be an integer such that $1\leq k<d$. Then $\gammakL(G)=\gamma_{1,d-k}(G)$.
\end{corollary}

In the special case when $G$ is a connected $d$-regular graph and $k=d-1$, Corollary \ref{d-regular} implies that $\gamma_{d-1}^{\mathrm L}(G)=\gamma_{1,1}(G)=\gt(G)$. 
For $d=2$, the exact result for the total domination number is given in \cite{MA_Henning}, which yields the following corollary.

\begin{corollary}\label{2-regular}
Let $C_n$ be a cycle on $n\geq 3$ vertices. Then 
$\gammaenaL(C_n)=\left \lfloor \frac n2 \right \rfloor+ \left \lceil \frac n4 \right \rceil - \left \lfloor \frac n4 \right \rfloor.$
\end{corollary}

It is known that for a cubic graph $G$ it holds  $\frac n3
\leq \gamma_t(G)\leq
\frac n2$, see results in 
\cite{Henning}.
To see that both bounds are tight, consider the prism graph over $C_n$, denoted by $C_n\Box K_2$, and obtained from two copies of a cycle $C_n$ by joining the corresponding vertices. It is easy to verify that $\gamma_t(C_3\Box K_2)=2$ and  $\gamma_t(C_4\Box K_2)=4$.

Furthermore, in \cite{fakhran} the tight bounds $\frac{n}{2}\leq \gamma_{1,2}(G)\leq\frac{3n}{4}$ for a cubic graph $G$ were established. Thus, by Corollary \ref{d-regular} we arrive at the following result.

\begin{corollary}\label{3-regular}
Let $G$ be a connected cubic graph on $n\geq 3$ vertices. Then $\frac n3 \leq \gamma_2^{\mathrm L}(G)\leq \frac{n}{2}$ and $\frac{n}{2}\leq\gammaenaL(G)\leq \frac{3n}{4}$, and the bounds are tight.
\end{corollary}

\section{$1$-limited domination vs.~packing}
\label{sec:1L}

It is of independent interest to consider the $k$-limited domination for a fixed value of $k$. In this section we present some observations for the basic case when $k=1$. 

We already know that in the case of cubic graphs $1$-limited domination coincides with $(1,2)$-domination. In what follows, we consider general graphs and establish the upper bounds on $1$-limited domination number in terms of the packing number of a graph, thereby improving the upper bound previously given in Theorem \ref{bounds}.

Recall that a subset $S\subseteq V(G)$ of a graph $G$ is a \textit{packing} in $G$ if for every distinct vertices $u,v\in S$ we have $N[u]\cap N[v]=\emptyset$. Equivalently, we can say $S$ is a packing if the vertices in $S$ are
pairwise at distance at least $3$ in $G$, or that $|N[v]\cap S|\leq 1$ for every $v\in V(G)$. The packing number $\rho(G)$ is the maximum cardinality of a packing. A packing of order $\rho(G)$ will be referred to as a \textit{$\rho(G)$-set}. It is easy to see that $\rho(G)\leq \gamma(G)$, see \cite{HHH3}. From the classical bound $\rho(G)\le \frac{n}{\delta(G)+1}$ for connected graphs, we obtain the general estimate $1 \le \rho(G) \le \lfloor \frac{n}{2} \rfloor$, 
see~\cite{klimited}.

\begin{theorem} \label{packing}
Let $G$ be a connected graph of order $n\geq 3$. Then $$\Big\lceil \frac{n}{2} \Big\rceil \leq \gammaenaL(G)\leq n-\rho(G).$$
The upper bound is attained if and only if there exists a $\rho(G)$-set $P\subseteq V(G)$ such that its complement is a $\gammaenaL(G)$-set.
\end{theorem}

\begin{proof}
Let $G$ be a connected graph of order $n\geq 3$.
Using Theorem~\ref{bounds}, and Ore's classical bound $\gamma(G)\le \frac n2$ for isolate-free graphs~\cite{ore}, we derive that $\gammaenaL(G)\geq  \lceil \frac{n}{2} \rceil$. 

To prove the upper bound, let $P$ be a maximum packing of $G$, i.e.~$|P|=\rho(G)$, and let $D$ be its complement in $G$, that is $D=V(G)-P$. Let $u$ be an arbitrary vertex in $P=V(G)-D$. Clearly $u$ has no neighbor in $P$, thus it has to be adjacent to a vertex in $D$, so $D$ is a dominating set. Moreover, the 
$1$-limited constraint holds as $P$ is a packing.
Thus $D$ is a  $1$-limited dominating set, and we have
$\gammaenaL(G)\leq |D|=n-|P|= n-\rho(G)$.

To characterize graphs attaining the upper bound,  
let first $P$ be a $\rho(G)$-set and let $V(G)-P$ be a $\gammaenaL(G)$-set. By the reasoning above, the equality $\gammaenaL(G)=n-\rho(G)$ immediately follows.

Conversely, let $G$ be a graph with $\gammaenaL(G)=n-\rho(G)$.
Assume $P$ is a $\rho(G)$-set. Then, by the claim above, $D=V(G)-P$ is a $1$-limited dominating set, and we have 
$n-\rho(G)=\gammaenaL(G)\leq |D|=n-|P|$. Thus $\gammaenaL(G)=|D|$, i.e.~$V(G)-P$ is a $\gammaenaL(G)$-set.
\end{proof} 

The lower bound in Theorem~\ref{packing} is clearly attained precisely for graphs that admit a $1$-limited dominating set $D$ with $|D|=\lceil \frac{n}{2} \rceil$.
Their structure can be described in more detail. To do so, we use the notation $G[D,S]$ for 
the bipartite spanning subgraph of $G$ with nonempty partite sets $D$ and $S$, 
whose edge set 
consists exactly of all edges of $G$ with one endvertex in $D$ and the other in $S$.

\begin{corollary}\label{cor-1}
Let $G$ be a connected graph of order $n\ge 3$. 
\begin{enumerate}
\item[(i)] If $n$ is even, then $\gammaenaL(G)=\frac n2$ if and only if there exist a partition $(D,S)$ of $V(G)$, such that $G[D,S]$ is a disjoint union of $\frac n2$ copies of $K_2$.
\item[(ii)] If $n$ is odd, then $\gammaenaL(G)=\lceil \frac{n}{2} \rceil$ if and only if there exist a partition $(D,S)$ of $V(G)$ with $|D|=\lceil \frac{n}{2} \rceil$, such that  $G[D,S]$ is either a disjoint union of $\lfloor \frac{n}{2} \rfloor$ copies of $K_2$ and one isolated vertex lying in $D$, or a disjoint union of $\lfloor \frac{n}{2} \rfloor -1$ copies of $K_2$ and a path $P_3$ with the central vertex lying in $S$.
\end{enumerate}
\end{corollary}

\begin{proof} Let $G$ be a connected graph  of order at least $3$ with $\gammaenaL(G)= \lceil \frac{n}{2} \rceil$, and let $D$ be a $\gammaenaL(G)$-set. Then the cardinality of $S=V(G)-D$ is clearly $\lfloor \frac{n}{2}\rfloor$ and $(D,S)$ is a partition of $V(G)$. Since $D$ is a $1$-limited dominating set, every vertex of $S$ has at least one neighbor in $D$, and each vertex of $D$ has at most one neighbor in $S$. If $n$ is even, this readily implies that $G[D,S]$ consists of $\frac{n}{2}$ independent edges. In the case when $n$ is odd, suppose that there exist two vertices in $S$ having at least two neighbors in $D$. But this implies there is a vertex in $D$ with two neighbors in $S$, a contradiction. Thus with the exception of one vertex in $S$, say $v$, every vertex has exactly one neighbor in $D$, and $v$ either has one or two neighbors in $D$, implying the desired structure of $G[D,S]$ also when $n$ is odd.

Now we prove the converse implications. In each of the cases  with respect to the given structure of $G[D,S]$ (one case if $n$ is even, and two cases if $n$ is odd), one can immediately see that $D$ is a $1$-limited dominating set with $|D|=\lceil \frac{n}{2} \rceil$. Thus Theorem~\ref{packing} implies that $\gammaenaL(G)=\lceil \frac{n}{2} \rceil$.
\end{proof}

Having characterized the graphs that attain the lower bound 
$\gammaenaL(G)=\lceil \frac n2 \rceil$, we next turn to the opposite 
extremal situation in which the upper bound 
$\gammaenaL(G)=n-\rho(G)$ is attained.  
By Theorem~\ref{packing}, this happens precisely when a maximum packing 
$P$ of $G$ has the property that its complement $V(G)-P$ is a 
$1$-limited dominating set.  
This condition significantly restricts the possible neighborhood 
structure between $P$ and the rest of the graph.

This becomes particularly restrictive when the packing number assumes its 
largest possible value, namely $\rho(G)=\lfloor \frac n2 \rfloor$.  
In the case of even order, $\rho(G)=\frac n2 $ forces an especially rigid 
structure.  
A result of Mojdeh et al.~\cite{packing} shows that the connected 
graphs with $\rho(G)=\frac n2 $ are exactly the \textit{corona graphs} 
$H\odot K_1$ obtained from a connected graph $H$ of order $\frac n2$ by 
attaching one pendant vertex to each of its vertices.  
Combining this structural description with Theorem~\ref{packing} yields 
a complete characterization of all connected graphs of even order 
for which the bounds 
$\left\lceil\frac n2\right\rceil \le \gamma^{\mathrm L}_1(G) \le n-\rho(G)$
coincide. This is summarized in the next corollary.

\begin{corollary}\label{n-sod}
Let $G$ be a connected graph of even order $n\ge 4$. Then the following
statements are equivalent:
\begin{enumerate}
\item[(i)] $\rho(G)=\frac n2$,
\item[(ii)] $\gammaenaL(G)=\rho(G)$,
\item[(iii)] $G$ is a corona graph.
\end{enumerate}
\end{corollary}

\begin{proof} The equivalence $(i)\Leftrightarrow(ii)$ is immediately derived from inequalities in 
Theorem~\ref{packing}, while the equivalence $(i)\Leftrightarrow(iii)$ 
follows from the characterization of graphs with $\rho(G)=\frac n2$ 
established in~\cite{packing}.
\end{proof}

In the above argumentation we have used the result of Mojdeh et al., although we could prove Corollary \ref{n-sod} also via sequence of implications $(i) \Rightarrow(ii) \Rightarrow(iii) \Rightarrow (i)$  by using just Theorem \ref{packing}. Such an approach also makes it possible to treat the case when 
$n$ is odd, which, to the best of our knowledge, has not yet been addressed in the literature. For this purpose we introduce the notion of a near-corona graph.

A connected graph $G$ of odd order $n$ is a \emph{near-corona graph} if it contains an independent set 
$P$ of size $\lfloor\frac n2 \rfloor$ such that the closed neighborhoods $N[v]$, 
$v \in P$, are pairwise disjoint, and at most one vertex of $P$ has degree $2$ 
while all the remaining vertices of $P$ are pendant.

\begin{corollary}\label{n-lih}
Let $G$ be a connected graph of odd order $n\ge 3$. Then the following
statements are equivalent:
\begin{enumerate}
\item[(i)] $\rho(G)=\lfloor\frac n2 \rfloor$,
\item[(ii)] $\gammaenaL(G)=\rho(G)+1$,
\item[(iii)] $G$ is a near-corona graph.
\end{enumerate}
\end{corollary}

\begin{proof}
The implication $(i)\Rightarrow (ii)$ is a direct corollary of Theorem \ref{packing}. 

To prove $(ii)\Rightarrow (iii)$, let $\gammaenaL(G)=\rho(G)+1$. As $\rho(G)\leq \lfloor\frac n2 \rfloor$, we have $\gammaenaL(G)\leq \lfloor\frac n2 \rfloor +1 = \lceil\frac n2 \rceil$. Then, again  by Theorem \ref{packing}, we infer 
$\gammaenaL(G)=\lceil\frac n2 \rceil$ and $\rho(G)=\lfloor\frac n2 \rfloor$, meaning that $G$ is a graph whose structure satisfies the characterization of the upper bound in Theorem~\ref{packing}. Thus in $G$ there exists a $\rho(G)$-set $P$, such that its complement $V(G)-P$ is a $1$-limited dominating set. As $P$ is a packing, we already know that $P$ is an independent set. Since $V(G)-P$ (which contains $\lceil \frac n2 \rceil$ vertices) is a dominating set, every vertex in $P$ is adjacent to a vertex in 
$V(G)-P$, and the $1$-limited constraint further implies
that $G$ has the structure of a near-corona graph.

Finally, to see that $(iii)\Rightarrow (i)$ holds, let $G$ be a near-corona graph. 
Then it contains an independent set 
$P$ of size $\lfloor\frac n2 \rfloor$ such that the closed neighborhoods $N[v]$ where 
$v \in P$, are pairwise disjoint, which already means that $\rho(G)\geq \lfloor \frac n2 \rfloor$.
Since the opposite inequality holds in general, the result follows.
\end{proof}

In particular, Corollaries~\ref{n-sod} and~\ref{n-lih} together yield a complete characterization of graphs with $\rho(G)=\lfloor\frac n2 \rfloor$.

The latter represents one extreme where the packing number is as large as possible and, as we have seen, forces a highly constrained structure.  
At the opposite end of the spectrum lies the case $\rho(G)=1$, where the packing number is minimal. This condition imposes 
a very different structural behavior.  
We therefore next analyze graphs with $\rho(G)=1$, beginning with the simplest
subclass, namely graphs containing a \emph{universal vertex}, i.e.~a vertex adjacent to every other vertex of the graph.

\begin{lemma}\label{k=1:universal}
Let $G$ be a connected graph on $n\geq 3$ vertices containing a universal vertex. Then $\gammaenaL(G)=n-1$.  
\end{lemma}

\begin{proof}
Let $G$ be a graph that satisfies the assumptions of the lemma. Since $G$ contains a universal vertex, we have $\rho(G)=1$. Thus by Theorem \ref{packing} (and also Theorem \ref{bounds}) we have $\gammaenaL(G)\leq n-1$.

To prove the opposite inequality, let $D \subseteq V(G)$ be a $\gammaenaL(G)$-set. 
First consider the case when a universal vertex $u$ lies in $D$. Then at most one of its neighbors can lie outside $D$, thus $|D| \geq n-1$. On the other hand, if $u \notin D$, there exists some $x \in D$ that dominates $u$. Suppose there is another vertex $y \notin D$. Clearly, $y$ cannot be adjacent to $x$, which means there must exist $z \in D$ that dominates $y$. However, since $u$ is universal, $z$ has at least two neighbors outside $D$, a contradiction. Hence, in this case only $u$ can lie outside $D$, implying that $|D| = n-1$.
\end{proof}

The converse of Lemma~\ref{k=1:universal} does not hold. For instance, the graph shown in Figure~\ref{S2} satisfies $\gamma_1^{\mathrm L}(G)=5=n-1$, yet it has no universal vertex.

\begin{figure}[H]
\begin{center}
\begin{tikzpicture}[scale=0.4]
			
\node [My Style, name=a]   at (0,0) {};		
\node [My Style, name=b]   at (5,0) {};	
\node [My Style, name=c]    at (2.5,5) {};	
\node [My Style2, name=x1]   at (2.5,0) {};	
\node [My Style, name=x2]   at (3.75,2.5) {};	
\node [My Style, name=x3]   at (1.25,2.5) {};	
		
\draw[thick] (a)--(x1)--(b)--(x2)--(c)--(x3)--(a);
\draw[thick] (x1)--(x2)--(x3)--(x1);




\end{tikzpicture}
\end{center}
\caption{A graph $G$ with the black vertices forming a $\gammaenaL(G)$-set.}
\label{S2}
\end{figure}
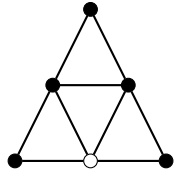

This example also suggests that, in order for a graph $G$ to satisfy 
$\gamma^{\mathrm L}_1(G) = n-1$, its diameter must be small. 
By Theorem~\ref{packing} we have $\gamma^{\mathrm L}_1(G)\le n-\rho(G)$, 
so the condition $\gamma^{\mathrm L}_1(G)=n-1$ forces $\rho(G)=1$. 
Moreover, it is known that for a graph $G$ of order $n$ we have 
$\rho(G)=1$ if and only if $\diam(G)\le 2$, see \cite{klimited}. 
Combining these facts yields the following corollary.

\begin{corollary}\label{diameter2}
Let $G$ be a connected graph of order $n\geq 3$. If $\gammaenaL(G)=n-1$, then $\diam(G)\leq 2$.
\end{corollary}

Before presenting the next result, leading to a characterization of graphs with $\gammaenaL(G)=n-1$, we recall that a \emph{pendant edge} is an edge with one endvertex of degree $1$. An edge that is not pendant will be referred to as a \emph{non-pendant edge}. 

\medskip

\begin{lemma}\label{pendant}
Let $G$ be a connected graph on $n\geq 3$ vertices. 
If $\gamma^{\mathrm L}_1(G)=n-1$, then $\mathrm{diam}(G)\le 2$,
and every non-pendant edge of $G$ lies in a triangle or $G$ contains a universal vertex.
\end{lemma}

\begin{proof}
Suppose that $\gammaenaL(G)=n-1$. 
By Corollary~\ref{diameter2} we have $\diam(G)\leq 2$.
Assume that $G$ has no universal vertex. We will prove that every non-pendant edge of $G$ is contained in a triangle. Suppose, to the contrary, that there exists a non-pendant edge $uv \in E(G)$ that does not belong to any triangle. Then $u$ and $v$ both have at least one neighbor in the set $D_{u,v}=V(G)-\{u,v\}$ which is thus clearly a dominating set of $G$. Moreover, since $uv$ does not lie in a triangle, no vertex of $G$ is adjacent to both $u$ and $v$. Hence $D_{u,v}$ is a $1$-limited dominating set of $G$. This yields $\gammaenaL(G) \leq |D_{u,v}| = n - 2$, a contradiction. Therefore, every non-pendant edge of $G$ must lie in a triangle, or $G$ contains a universal vertex.
\end{proof}

Despite the detailed conditions of Lemma \ref{pendant}, the converse statement is again not true. The graph $H$ shown in Figure \ref{rozica} has diameter $2$, and all its edges (none of them pendant) lie in a triangle, yet $\gammaenaL(G)\leq n-3$.

\begin{figure}[H]
\vspace{-1cm} 
\begin{center}
\begin{tikzpicture}[scale=0.25]
            
\node [My Style2, name=a]   at (0,1) {};		
\node [My Style2, name=b]   at (5,1) {};	
\node [My Style2, name=c]    at (2.5,4.5) {};	
\node [My Style, name=x1]   at (-2.2,2.6) {};	
\node [My Style, name=x2]   at (0,-2) {};	
\node [My Style, name=y1]   at (7.2,2.6) {};	
\node [My Style, name=y2]   at (5,-2) {};
\node [My Style, name=z1]   at (0.1,6.2) {};		
\node [My Style, name=z2]   at (4.9,6.2) {};	
		
\draw[thick] (a)--(b)--(c)--(a);
\draw[thick] (a)--(x1)--(x2)--(a);
\draw[thick] (b)--(y1)--(y2)--(b);
\draw[thick] (c)--(z1)--(z2)--(c);
\draw[thick] (z1)--(x1);
\draw[thick] (x2)--(y2);
\draw[thick] (y1)--(z2);
\draw[thick] (z2) to[out=120, in=110, looseness=1.8] (x1);
\draw[thick] (x1) to[out=250, in=230, looseness=1.8] (y2);
\draw[thick] (y2) to[out=10, in=350, looseness=1.8] (z2);

\end{tikzpicture}
\end{center}
\vspace{-1cm} 
\caption{The black vertices represent a $1$-limited dominating set $D$ of the given graph $H$, with $|D|=n-3$.}
\label{rozica}
\end{figure}
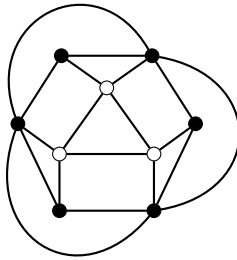

Recall that for a graph $G$, a set $S \subseteq V(G)$ and $s \in S$, the set
$$\epn(s,S)=\{v \in V(G)-S; N(v) \cap S = \{s\}\}
$$
is called the \textit{$S$-external private neighborhood} of $s$ (with respect to $S$), \cite{HHH1}.
One may observe that in the graph $H$ from Figure~\ref{rozica}, 
the set of white vertices 
$S$ has the property that for each $s\in S$, $\epn(s,S)\ne \emptyset$.
As we shall see in the next theorem, the existence of such a set $S$ of 
cardinality at least $3$ is one of the key structural reasons why 
$\gammaenaL(G)$ drops below $n-1$.

\begin{theorem}\label{zg-kar}
Let $G$ be a connected graph on $n\geq 3$ vertices. Then $\gammaenaL(G)=n-1$ if and only if
\begin{enumerate}
\item[(i)] $\mathrm{diam}(G)\leq 2$,
\item[(ii)] every non-pendant edge of $G$ lies in a triangle, or $G$ has a universal vertex, and
\item[(iii)] $G$ contains no subset $S \subseteq V(G)$ with $|S|\ge 3$
such that $\epn(s,S)\neq\emptyset$ for every $s\in S$, and $N(S)=\displaystyle \bigcup _{s\in S}\epn(s,S)$.
\end{enumerate}
\end{theorem}

\begin{proof}
Suppose first that $\gammaenaL(G)=n-1$. 
Then conditions $(i)$ and $(ii)$ follow from Lemma \ref{pendant}. To prove condition $(iii)$, suppose to the contrary, that $G$ contains a subset $S \subseteq V(G)$ with $|S|\ge 3$
such that $\epn(s,S)\neq\emptyset$ for every $s\in S$, and $N(S)= \bigcup _{s\in S}\epn(s,S)$.
Clearly, $\epn(s,S)\cap \epn(t,S)=\emptyset$ for any different vertices $s$ and $t$ from $S$. Thus each vertex in $N(S)$ is adjacent to precisely one vertex in $S$. 
Now, let $D=V(G)-S$.
Since $\epn(s,S)\neq\emptyset$ for every $s\in S$, $s$ has a neighbor in $D$, and so $D$ is a dominating set. Moreover, each vertex $d\in D$, that has a neighbor in $S$, lies in $N(S)$ and therefore
has exactly one neighbor in $S=V(G)-D$. 
Thus $D$ is a $1$-limited dominating set, and $\gammaenaL(G)\leq |D|=n-|S|\leq n-3$, a contradiction.

To prove the converse, assume that $G$ satisfies conditions $(i)-(iii)$, and let $D$ be a $\gammaenaL(G)$-set. 
If $G$ has a universal vertex, the claim follows from Lemma \ref{k=1:universal}. Hence in the sequel we assume, that $G$ has no universal vertex. 

Suppose that $|D|\leq n-2$, and let $S=V(G)-D$.
Then there exist distinct vertices $s_1, s_2\in S$. Since $\text{diam}(G)\leq 2$, either $s_1s_2\in E(G)$ or $d(s_1,s_2)=2$.  
First, let $s_1s_2\in E(G)$. If $\deg(s_i)=1$ for some $i\in\{1,2\}$, then $s_i$ does not have a neighbor in $D$, a contradiction. Hence, $s_1s_2$ is a non-pendant edge, and by the condition $(ii)$ there exists a vertex $s_3\in V(G)$ adjacent to both $s_1$ and $s_2$. Clearly $s_3\notin D$, since otherwise $s_3$ would dominate two vertices outside $D$, violating the $1$-limited constraint. 
In the second case, when $d(s_1,s_2)=2$, there exists $s_3\in V(G)$ such that $s_1s_3s_2$ is a path of length $2$, immediately implying that $s_3\in S$. Thus $|S|\geq 3$. 

Since $D$ is a dominating set,
every vertex $s\in S$ has at least one neighbor in $D$, say  $d$, and since $D$ is $1$-limited, $d$ has at most one neighbor
in $S$. Thus $N(d)\cap S=\{s\}$ and so $d\in \epn(s,S)$, implying $\epn(s,S)\neq\emptyset$
for every $s\in S$.
Moreover, it is straightforward to see that the sets 
$N(S)=\{d\in D ; N(d)\cap S\neq\emptyset\}$ and $ \bigcup_{s\in S}\epn(s,S)$ are equal.
But since $|S|\ge 3$, we now have a contradiction with condition $(iii)$.

Consequently, $|D| \ge n - 1$, while the reverse inequality is guaranteed by Theorem~\ref{bounds}, yielding $\gamma^{\mathrm L}_1(G) = n - 1$.
\end{proof}

We now apply Theorem~\ref{zg-kar} to the class of trees. Since a tree contains no cycle, no edge can lie in a triangle. Therefore, condition $(ii)$ of the theorem implies that if a tree $T$ satisfies $\gammaenaL(T)=n-1$, then $T$ must contain a universal vertex. But the only tree with a universal vertex is a star graph. Conversely, if $T$ is a star, then $\gammaenaL(T)=n-1$ follows by Proposition \ref{families}. We have therefore shown the following.

\begin{corollary}\label{k=1:star}
Let $T$ be a tree on $n\geq 3$ vertices. Then $\gammaenaL(T)=n-1$ if and only if $T$ is a star graph.  
\end{corollary}



\section{Conclusion}

In this paper we have introduced and studied the concept of \emph{$k$-limited domination}, motivated by applications where dominating vertices have limited capacity and cannot be overloaded by too many external neighbors, as well  as theoretical considerations that extend the scope 
of classical domination theory.

We established tight general bounds for $\gamma^{\mathrm L}_k(G)$ and analyzed their 
extremal behavior. For the special case $k=1$, we uncovered a Gallai-type relation 
$\gamma^{\mathrm L}_1(G)+\rho(G)\le n$, linking $1$-limited domination to the packing number. Examining the equality cases led to a complete structural 
characterization of connected graphs satisfying $\gamma^{\mathrm L}_1(G)=n-1$.
More importantly, the connection between these two concepts has enabled us to fully characterize the graphs that attain the upper bound $\lfloor \frac{n}{2} \rfloor$ for the packing number $\rho(G)$, which until now was known only for graphs with an even number of vertices.

A further contribution is the connection between $(1,t)$-domination and $k$-limited domination. 
For $d$-regular graphs we proved the identity
$\gamma^{\mathrm L}_k(G)=\gamma_{1,d-k}(G)$,
which suggests that 
$k$-limited domination may offer new insights on
$(1,t)$-domination, including open problems in~\cite{trilobiti} and~\cite{fakhran}.
In addition, for some basic families of graphs (paths, cycles, stars, complete graphs and complete bipartite graphs) we derived closed formulae for $\gamma^L_k(G)$ for all interesting values of $k$. 

The results obtained in the paper open a wide range of further questions. A natural direction 
is to characterize graphs attaining the bounds in Theorem~\ref{bounds}.
In an attempt to improve the upper bound, in Proposition~\ref{Delta-v-zg} we have considered 
a special degree condition that turned out to be essential
for the constructive proof:
when augmenting a $\gamma$-set $D$ with the neighbors of a vertex $u\in D$ to enforce the
$k$-limited constraint, newly added vertices of degree greater than $k+1$ may themselves violate
the constraint, potentially triggering a cascade of additions that could exceed $\Delta(G)-k$
per dominator. The assumption $\deg(w)\le k+1$ for the selected neighbors prevents
such cascades and ensures the bound 
$\gammakL(G) \le (\Delta(G)-k+1)\,\gamma(G)$.
Whether the bound is valid in full generality, without the degree condition, remains an interesting open question.

Determining 
$\gamma^{\mathrm L}_k(G)$ for additional graph families, exploring the algorithmic and 
complexity aspects of the parameter, and investigating its behavior in various graph
products are among the natural topics for further study.

Given its structural depth, its interrelations with different domination concepts, and its potential for modeling real-world network constraints, we expect $k$-limited domination to become a valuable tool in domination theory and a driving
force for future developments in this area.


\vskip 1pc \noindent{\bf Acknowledgments.} 

D.B., G.R.~and A.T.~acknowledge the financial support of the Slovenian Research and Innovation Agency ARIS (research core funding No.\ P2-0065 [D.B.], P1-0288 [G.R.], P1-0297 and J1-70016 [A.T.], and the bilateral project between Slovenia and Montenegro entitled {\it Modern Domination Concepts and Their Application}, project No. BI-ME/25-27-001). 
\v{Z}.K.V.~acknowledges the financial support of the  Ministry of Education, Science and Innovation of Montenegro through the grant 12/2-633/25-4916 and the bilateral project no. 0604-082/24-1946.

\bibliographystyle{plain}
\bibliography{references}

\end{document}